\documentclass[12pt,twoside]{article}
\usepackage{mathrsfs}

\usepackage[centertags]{amsmath}
\usepackage{amsfonts}
\usepackage{amssymb}
\usepackage{graphicx}
\usepackage{amsthm}
\usepackage{newlfont}

\newtheorem{theorem}{Theorem}[section]

\newtheorem{lemma}[theorem]{Lemma}
\newtheorem{proposition}[theorem]{Proposition}

\theoremstyle{remark}

\numberwithin{equation}{section}

\begin{document}
\title{Acute Triangulations of the Cuboctahedral Surface}
\author{Xiao Feng  \ \ \ \ \ \ Liping Yuan\thanks {Corresponding author}}
\date{}
 \maketitle

\begin{center}
{ College of Mathematics and Information Science,

 Hebei Normal
University, 050016 Shijiazhuang, China.}

lpyuan@mail.hebtu.edu.cn.
\end{center}
\begin{abstract}

In this paper we prove that the surface of the cuboctahedron  can be
triangulated into 8 non-obtuse triangles and  12 acute triangles.
Furthermore, we show that both bounds are the best possible.

\end{abstract}

\section{Introduction}

A {\em triangulation} of a two-dimensional space means  a collection
of (full) triangles covering the space, such that the intersection
of any two triangles is either empty or consists of a vertex or of
an edge. A triangle is called {\it geodesic} if all its edges are
{\it segments}, i.e., shortest paths between the corresponding
vertices. We are interested only in {\it geodesic triangulations},
all the members of which are, by definition, geodesic triangles. The
number of triangles in a triangulation is called its {\it size}.

In rather general two-dimensional spaces, like Alexandrov surfaces,
two geodesics starting at the same point determine a well defined
angle. Our interest will be focused on triangulations which are {\it
acute} (resp. {\it non-obtuse}), which means that the angles of all
geodesic triangles are
 smaller (resp. not greater) than $\frac{\pi}{2}$.

 The discussion
of acute triangulations has one of its origins in a problem of
Stover reported in 1960 by Gardner in his Mathematical Games section
of the {\it Scientific American} (see \cite{margard1},
\cite{margard2}, \cite{margard}). There the question was raised
whether a triangle with one obtuse angle can be cut into smaller
triangles, all of them acute. In the same year, independently,
Burago and Zalgaller \cite{bura} investigated in considerable depth
acute triangulations of polygonal complexes, being led to them by
the problem of their isometric embedding into $\mathbb{R}^3$.
However, their method could not give an estimate on the number of
triangles used in the existed acute triangulations.  In 1980,
Cassidy and Lord ~\cite{ccgl} considered acute triangulations of the
square. Recently,  acute triangulations of quadrilaterals
\cite{hmarqua}, trapezoids \cite{yuan-trapezoid}, convex
quadrilaterals \cite{convex-qua}, pentagons \cite{yuan-pentagon} and
general polygons \cite{hmarpoly, yuanpolygon} have also been
considered.

On the other hand, compact convex surfaces have also been
triangulated. Acute and non-obtuse triangulations of all Platonic
surfaces, which are  surfaces  of the five well-known  Platonic
solids, have been investigated in \cite{hiz01}, \cite{izam20}, and
\cite{izam12}. Recently, Saraf \cite{saraf} considered the acute
triangulations of the polyhedral surfaces again, but there is
 still no estimate on the size of the existed acute triangulations.
 Maehara \cite{hmar-polyhedral} considered the proper acute triangulation of a polyhedral
 surface and obtained an upper bound of the size of the triangulation, which is determined by the length of the longest edge,
  the minimum value of the geodesic distance from a vertex to an edge that is not incident to the vertex, and
   the measure of the smallest face angle in the given polyhedral surface.
Furthermore, some other well-known surfaces have  also been acutely
triangulated, such as flat M\"obius strips \cite{yuan-mobius} and
flat tori \cite{itoh-yuan}.

 Combining all the known results for the polyhedral surfaces mentioned above,  we are  motivated to
investigate the non-obtuse and acute triangulations of the surfaces
of the Archimedean solids. In this paper we consider the surface of
 the Archimedean solid cuboctahedron, which is a convex polyhedron with eight
triangular faces and six square faces. It has 12 identical vertices,
with two triangles and two squares meeting at each, and 24 identical
edges, each separating a triangle from a square. For the sake of
convenience, let $\mathcal{C}$ denote the surface of the
cuboctahedron with side length 1.
 Let $\mathscr{T}$ denote  an  acute
 triangulation of $\mathcal{C}$ and  $\mathscr{T}_0$  a non-obtuse
 triangulation of $\mathcal{C}$. Let  $|\mathscr{T}|$ and
 $|\mathscr{T}_0|$
 denote  the size of $\mathscr{T}$ and $\mathscr{T}_0$ respectively. We prove that the best possible bounds for
$|\mathscr{T}|$ and $|\mathscr{T}_{0}|$ are $12$ and $8$
respectively.

\section{Non-obtuse triangulations}
\begin{theorem}\label{thm1}
The surface of the cuboctahedron admits a non-obtuse triangulation
with $8$ triangles and no non-obtuse triangulation with fewer
triangles.
\end{theorem}
\begin{proof}
Fig.  \ref{fig-1} describes the unfolded surface $\mathcal{C}$. We
fix two vertices $a$ and $b$, which are the vertices of a diagonal
of a square face on $\mathcal{C}$. Let $a'$, $b'$ be the antipodal
vertices of $a$, $b$ respectively. There are six geodesics  from $a$
to $a'$ and $b$ to $b'$.
\begin{figure}[htbp]
\begin{center}
 \includegraphics*[height=5cm]{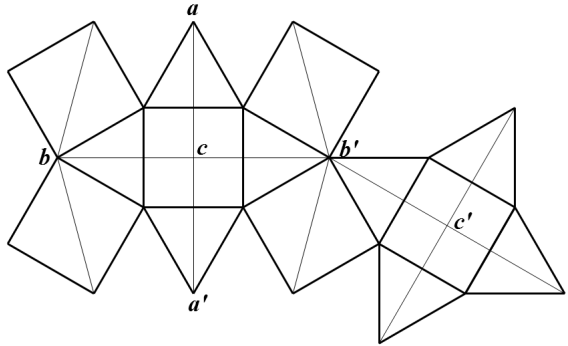}
 \end{center}
 \caption{ A non-obtuse triangulation of $\mathcal{C}$. }
\label{fig-1}
\end{figure}We choose those two passing through two
triangular faces and one square face. Denote the two intersection
points of the geodesics $aa'$ and $bb'$ chosen above by $c$ and
$c'$. Clearly, $c$ and $c'$ are an antipodal pair of vertices on
$\mathcal{C}$. Draw the segments from $a$ (resp. $a'$) to $b$ and
$b'$. Thus $\mathcal{C}$ is triangulated into $8$ non-obtuse
triangles:
 $abc, ab'c, abc', ab'c', a'bc,
a'b'c, a'bc', a'b'c'$.

Indeed, noticing that all of those eight triangles are congruent, we
only need to show that the triangle $abc$ is non-obutse. By the
construction we know that $aa'$ is orthogonal to $bb'$. So $\angle
acb=\frac{\pi}{2}$. Further, it is clear that $\angle abc=\angle
bac=\frac{5\pi}{12}$.

We prove now that for any  non-obtuse triangulation
$\mathscr{T}_{0}$ of $\mathcal{C}$, we always have
$|\mathscr{T}_{0}|\geq 8$. If not, then we have
$|\mathscr{T}_{0}|=4$ or $|\mathscr{T}_{0}|=6$. If
$|\mathscr{T}_{0}|=4$, then $\mathscr{T}_{0}$ has $(4\times3)/2=6$
edges and, by Euler's formula, $6-4+2=4$ vertices. So
$\mathscr{T}_{0}$ is isomorphic to $K_{4}$; If
$|\mathscr{T}_{0}|=6$, then $\mathscr{T}_{0}$ is isomorphic to the
1-skeleton of the double pyramid over the triangle. In both cases
there are vertices with degree 3. However, at each vertex of
$\mathcal{C}$ the total angle is $\frac{5\pi}{3}$, so each vertex in
$\mathscr{T}_{0}$ has degree at least 4.  Clearly, each other vertex
of $\mathscr{T}_{0}$ also has degree at least 4. Thus we obtain a
contradiction.

The proof is complete.
\end{proof}
\section{Acute triangulations}
\begin{theorem}\label{thm2}
The surface of the cuboctahedron  admits an acute triangulation with
$12$ triangles.
\end{theorem}
\begin{proof} Let $a'$, $b'$, $c'$ and $d'$ be four distinct
vertices of the cuboctahedron such that
$|a'b'|=|b'c'|=|c'd'|=|d'a'|=\sqrt{2}$, where $|pq|$ denotes the
intrinsic distance on the surface $\mathcal{C}$ between two points
$p$ and $q$.  Clearly, the four segments $a'b'$, $b'c'$, $c'd'$ and
$d'a'$ determine a cycle which decomposes $\mathcal{C}$ into two
regions $\mathcal {C}_{1}$ and $\mathcal {C}_{2}$. Take a vertex $a$
(resp. $b, c, d$) adjacent to both $a'$ (resp. $b', c', d'$) and
$b'$ (resp. $c', d', a'$) such that $a, c\in \mathcal {C}_{1}$,
$b,d\in\mathcal {C}_{2}$. Take a point $a^{*}$ (resp. $b^{*}, c^{*},
d^{*}$) on $a'b'$ (resp. $b'c', c'd', d'a'$) such that $\angle
a'aa^{*}$ (resp. $\angle b'bb^{*}$, $\angle c'cc^{*}$, $\angle
d'dd^{*}$) =$\frac{\pi}{6}$.

We get a triangulation of $\mathcal{C}$ with 12 triangles:

$a^{*}ab^{*}, a^{*}b^{*}b, a^{*}bd, a^{*}dd^{*}, a^{*}d^{*}a,
b^{*}bc^{*}, b^{*}c^{*}c, b^{*}ca, c^{*}cd^{*}, c^{*}d^{*}d,
c^{*}db, d^{*}ca$.

\begin{figure}[htbp]
\begin{center}
 \includegraphics*[height=6cm]{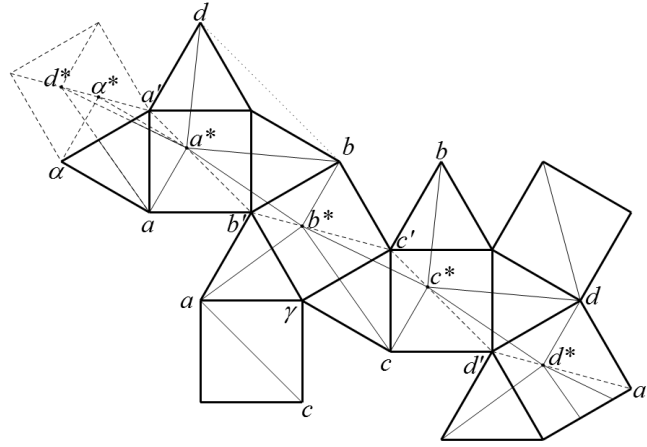}
 \end{center}
 \caption{ An acute triangulation of $\mathcal {C}$. }
\label{fig-2}
\end{figure}

There are two shortest paths from $a^{*}$ to $b^{*}$ (resp. $c^{*}$
to $d^{*}$); here we choose the path in $\mathcal {C}_{2}$. There
are two shortest paths from $b^{*}$ to $c^{*}$ (resp. $d^{*}$ to
$a^{*}$); here we choose the path in $\mathcal {C}_{1}$, see Fig.
\ref{fig-2}.

Indeed, the values of the angles around $a^{*},b^{*},c^{*},d^{*}$
(resp. $a,b,c,d$) are entirely the same. So we only need to consider
the angles around $a^{*}$ and $a$ respectively.

Firstly, we consider the angles around $a^{*}$.

In the triangle $a^{*}b'b$, $\angle
a^{*}b'b=\frac{\pi}{4}+\frac{\pi}{3}=\frac{7\pi}{12}>\frac{\pi}{2}$,
which implies that $\angle b^{*}a^{*}b<\angle
b'a^{*}b<\frac{\pi}{2}$.

In Fig. \ref{fig-2} the planar  circle $C$ with diameter $bd$ (the
dot line-segment) passes through the midpoint of $a'b'$, say, $x'$.
So, the segment $a'b'$ is tangent to $C$ at $x'$. Since $a^{*}\in
a'b'$ and $a^{*}\neq x'$, we have $\angle ba^{*}d<\frac{\pi}{2}$.

 In the quadrilateral $a^{*}da'd^{*}$, $\angle
da'd^{*}=\frac{\pi}{3}+\frac{\pi}{2}+\frac{\pi}{3}+\frac{\pi}{4}=\frac{17\pi}{12}$,
which implies that $\angle a^{*}d^{*}a'+\angle d^{*}a^{*}d+\angle
a^{*}da'=\frac{7\pi}{12}$. However, $\angle a^{*}da'>\angle
ada'=\frac{\pi}{12}$. Therefore, $\angle d^{*}a^{*}d<\angle
d^{*}a^{*}d+\angle
a^{*}da'<\frac{7\pi}{12}-\frac{\pi}{12}=\frac{\pi}{2}$.

Denote by $\alpha$ the vertex adjacent to both $a$ and $a'$. Take a
point $\alpha^{*}\in d'a'$ such that $\angle \alpha^{*}\alpha
a'=\frac{\pi}{6}$. Clearly,  the triangle $aa^{*}a'$ is congruent to
the triangle $\alpha\alpha^{*}a'$, so we have
$|\alpha^{*}a'|=|a^{*}a'|$, which implies that $\angle
a'a^{*}\alpha^{*}=\frac{\pi}{12}$. Then $\angle
\alpha^{*}a^{*}a=(\pi-\frac{\pi}{6}-\frac{\pi}{4})-\frac{\pi}{12}=\frac{\pi}{2}$.
Noticing that the distance from $d^{*}$ to $a'$ is further than that
from $\alpha^{*}$ to $a'$, we have $\angle d^{*}a^{*}a<\angle
\alpha^{*}a^{*}a=\frac{\pi}{2}$.

In the triangle $aa^{*}b'$, $\angle
aa^{*}b'=\pi-\frac{\pi}{3}-\frac{\pi}{4}=\frac{5\pi}{12}$. In the
triangle $a^{*}b'b^{*}$, $\angle
a^{*}b'b^{*}=\frac{\pi}{4}+\frac{\pi}{3}+\frac{\pi}{4}=\frac{5\pi}{6}$.
Noticing that
$|b'b^{*}|<\frac{1}{2}|b'c'|=\frac{1}{2}|a'b'|<|a^{*}b'|$, we have
$\angle b'a^{*}b^{*}<\angle b'b^{*}a^{*}$ and therefore $\angle
b'a^{*}b^{*}<\frac{\pi}{12}$. So we have $\angle
aa^{*}b^{*}<\frac{5\pi}{12}+\frac{\pi}{12}=\frac{\pi}{2}$.

Consider now the angles around $a$. It is clear that  $\angle
a^{*}ad^{*}<\angle a^{*}a\alpha=\frac{\pi}{2}$ and $\angle
a^{*}ab^{*}=\angle a^{*}ab'+\angle
b'ab^{*}<\frac{\pi}{3}+\frac{\pi}{6}=\frac{\pi}{2}$. Let $\gamma$
denote the vertex adjacent to both $c$ and $c'$. Then $\angle
b^{*}ac=\angle \gamma ac+\angle b^{*}a\gamma<\angle \gamma ac+\angle
ba\gamma=\frac{\pi}{4}+\frac{\pi}{4}=\frac{\pi}{2}$. Finally,
$\angle d^{*}ac=\angle d^{*}a\alpha+\angle \alpha
ac<\frac{\pi}{6}+\frac{\pi}{4}<\frac{\pi}{2}$.
\end{proof}

\section{No acute triangulation with fewer triangles}
Let $\mathscr{C}$ be the 1-skeleton of the cuboctahedron. The
graph-theoretic distance $d_\mathscr{C}(v,w)$ between the vertices
$v,\ w$ of $\mathscr{C}$ is called the {\it $\mathscr{C}$-distance}
between $v$ and $w$. Let $g(u,v)$ denote  a geodesic between two
points $u$ and $v$ on the surface $\mathcal{C}$. We start with the
following lemma.
\begin{lemma}\label{lem1}
Let $u$, $w_1$, $w_2$ be three vertices of $\mathcal{C}$. Then the
angle formed by $g(u, w_1)$ and $g(u, w_2)$ on $\mathcal{C}$ is
equal to $\frac{\pi}{12}i$, where $i\in \mathbb{Z}$ and $1\leq
i\leq20$.
\end{lemma}

\begin{figure}[htbp]
\begin{center}
 \includegraphics*[height=5cm]{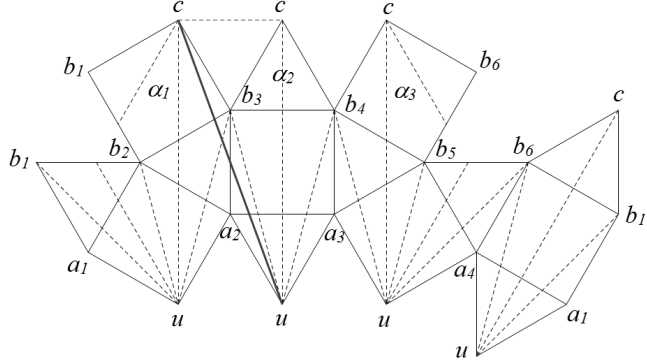}
 \end{center}
 \caption{ Geodesics starting from a vertex $u$. }
\label{fig-3}
\end{figure}

\begin{proof}
For any vertex $u$ of $\mathcal{C}$, consider all the segments from
$u$ to any other vertex $v$ (see Fig.  \ref{fig-3}). It is easy to
see that $d_{\mathscr{C}}(u,v)\leq 3$. If $d_{\mathscr{C}}(u,v)=1$,
then $v\in \{a_1, a_2, a_3, a_4\}$. Clearly $g(u,v)$ is an edge of
$\mathcal{C}$. If $d_{\mathscr{C}}(u,v)=2$, then $v\in \{b_1, b_2,
b_3, b_4, b_5, b_6\}$. Further, if $v=b_2$ or $v=b_5$, then $g(u,v)$
is a diagonal of a square face on the surface; if $v\in \{b_1, b_3,
b_4, b_6\}$, then there are two geodesics between $u$ and $v$. If
$d_{\mathscr{C}}(u,v)=3$, then $v=c$ and there are six geodesics
between $u$ and $v$. Please note that the solid line between $u$ and
$v$ in Fig. \ref{fig-3} is not a geodesic. Thus there are 20
geodesics starting from $u$ to any other vertex $v$ on
$\mathcal{C}$, and all of them divide the total angle around $u$
into 20 equal parts. Trivially, each part has angle
$\frac{\pi}{12}$. Thus the angle formed by $g(u, w_1)$ and $g(u,
w_2)$ on $\mathcal{C}$ is equal to $\frac{\pi}{12}i$, where $i\in
\mathbb{Z}$ and $1\leq i\leq20$.
  \end{proof}
\begin{lemma}\label{lem:cub-no 8}
There is no acute triangulation of $\mathcal{C}$ with $8$ triangles.
\end{lemma}
\begin{proof}
Suppose there exists an acute triangulation $\mathscr{T}$ of
$\mathcal{C}$ containing 8 triangles. By a method similar to that
used in the proof of Theorem \ref{thm1}, we know that $\mathscr{T}$
is isomorphic to the 1-skeleton of the regular octahedron, where all
the vertices have degree 4. Clearly each vertex of $\mathscr{T}$ is
a vertex of $\mathcal{C}$. By Lemma \ref{lem1}, it is easily seen
that any acute angle in $\mathscr{T}$ is $\frac{\pi}{12}i$,  where
$i=1, 2, 3, 4, 5$. Recall that the total angle at any vertex of
$\mathcal{C}$ is $\frac{5\pi}{3}$. Therefore the four angles around
each vertex of $\mathscr{T}$ are all isogonal and equal to
$\frac{5\pi}{12}$.

Now let $v_{1}, v_{2}$ be two adjacent vertices in $\mathscr{T}$.
Then in both of the triangles having side $g(v_{1},v_{2})$, all the
three angles are equal to $\frac{5\pi}{12}$. In the following we
show that in one of them, where the third vertex is denoted by
$v_3$, there is always a contradiction.

\begin{figure}[htbp]
\begin{center}
 \includegraphics*[height=3cm]{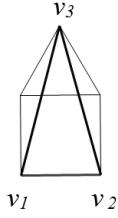}
 \end{center}
 \caption{ $d_{\mathscr{C}}(v_{1}, v_{2})=1$. }
\label{fig-4-1}
\end{figure}

\begin{figure}[htbp]
\begin{center}
 \includegraphics*[height=3cm]{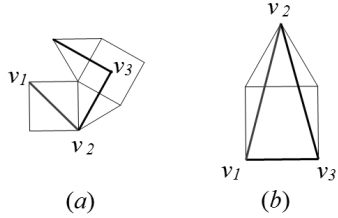}
 \end{center}
 \caption{ $d_{\mathscr{C}}(v_{1}, v_{2})=2$. }
\label{fig-4-2}
\end{figure}

There are three cases to consider.

$Case\ 1$.  $d_{\mathscr{C}}(v_{1}, v_{2})=1$. \ \

If $\angle v_{3}v_{1}v_{2}=\angle v_{3}v_{2}v_{1}=\frac{5\pi}{12}$,
then clearly we have $\angle v_{1}v_{3}v_{2}=\frac{\pi}{6}\neq
\frac{5\pi}{12}$, a contradiction, as shown in Fig.  \ref{fig-4-1}.

$Case\ 2$.  $d_{\mathscr{C}}(v_{1}, v_{2})=2$.

If $g(v_1,v_2)$ is a diagonal of a square face of $\mathcal {C}$,
then $\angle v_{3}v_{1}v_{2}=\angle v_{3}v_{2}v_{1}=\frac{5\pi}{12}$
forces $v_3$ not to be a vertex of $\mathcal{C}$ (see Fig.
\ref{fig-4-2}$(a)$), a contradiction. Otherwise, by the proof of
Lemma \ref{lem1} we may assume that $v_1$ is a corner of a square
face and $g(v_1,v_2)$ intersects the interior of the square face, as
show in Fig.  \ref{fig-4-2}$(b)$. Let $v_3$ be another corner of the
square face such that $\angle v_{3}v_{1}v_{2}=\frac{5\pi}{12}$. By
the proof of $Case1$, we know that $\angle
v_{1}v_{2}v_{3}=\frac{\pi}{6}$, a contradiction again.

\begin{figure}[htbp]
\begin{center}
 \includegraphics*[height=3cm]{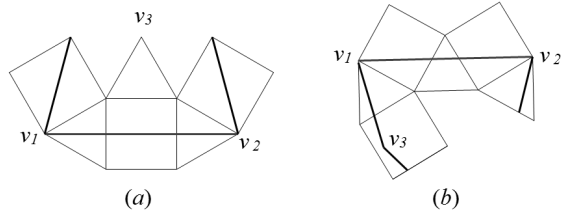}
 \end{center}
 \caption{ $d_{\mathscr{C}}(v_{1}, v_{2})=3$. }
\label{fig-4-3}
\end{figure}

$Case\ 3$. $d_{\mathscr{C}}(v_{1}, v_{2})=3$.

If $g(v_1,v_2)$ passes through two triangular faces and one square
face, then we consider the triangle $v_1v_2v_3$ lying above
$g(v_1,v_2)$, as shown in Fig.  \ref{fig-4-3}$(a)$. Clearly, $\angle
v_1v_2v_3=\angle v_2v_1v_3=\frac{5\pi}{12}$, but $\angle
v_1v_3v_2=\frac{5\pi}{6}$, a contradiction. If $g(v_1,v_2)$ passes
through one triangular face and two square faces, then we consider
the triangle $v_1v_2v_3$ lying below $g(v_1,v_2)$, as shown in Fig.
\ref{fig-4-3}$(b)$. It is easy to see that $\angle v_1v_2v_3=\angle
v_2v_1v_3=\frac{5\pi}{12}$ makes $v_3$ not be a vertex of
$\mathcal{C}$, which contradicts to the fact that each vertex of
$\mathscr{T}$ must be a vertex of $\mathcal{C}$.

The proof is complete.
\end{proof}
\begin{lemma}\label{lem:cub-no 10}
There is no acute triangulation of $\mathcal{C}$ with $10$
triangles.
\end{lemma}
\begin{proof}
Suppose that  there exists an acute triangulation $\mathscr{T}$ of
$\mathcal{C}$ containing 10 triangles. Then  $\mathscr{T}$ is
isomorphic to the 1-skeleton of the double pyramid over the
pentagon. So $\mathscr{T}$ contains a 5-cycle $C_{5}$ and all its
vertices have degree $4$. Clearly, the vertices of $C_{5}$ must be
the vertices of $\mathcal{C}$. For the sake of convenience, let
$V(C_5)$ denote the set of all vertices of $C_5$ and $E(C_5)$ denote
the set of all edges of $C_5$. Furthermore, we have the following
fact.

{\bf Fact.} The angles formed by any two adjacent edges of $C_{5}$
are between $\frac{2\pi}{3}$ and $\pi$.

If $u$, $v$ are two adjacent vertices of $\mathcal{C}$, we call $u$,
$v$ an {\it adjacent pair} of $\mathcal{C}$. In order to prove Lemma
\ref{lem:cub-no 10}, we prove the following properties about the
cycle $C_5$ mentioned above at first.
\begin{proposition}\label{prop:cub-no 10-1}
 $V(C_{5})$ contains at least two adjacent pairs of
$\mathcal{C}$.
\end{proposition}
\begin{proof} We first show that $V(C_{5})$ contains at least one adjacent pair of
$\mathcal{C}$. Suppose that $u\in V(C_5)$, as shown in Fig.
\ref{fig-3}. If $\{a_1, a_2, a_3, a_4\} \cap V(C_5) \neq \emptyset$,
then clearly $V(C_{5})$ contains an adjacent pair. If not, then the
other four vertices of $C_5$ come from $b_1, b_2, b_3, b_4, b_5,
b_6$ and $c$. It is easy to see that among those four vertices there
must be at least one adjacent pair of $\mathcal{C}$.

Now let $v_{1}, v_{2}\in V(C_{5})$ be an adjacent pair of
$\mathcal{C}$, as shown in  Fig.  \ref{fig-5}.   If $\{u_1, u_2,
u_3, u_4, u_5\} \cap V(C_5) \neq \emptyset$, then clearly $V(C_{5})$
contains another adjacent pair of $\mathcal{C}$ and the proposition
is proved. Otherwise,  the other three vertices of $C_5$ come from
the remained five vertices $\overline{u}_1, \overline{u}_2,
\overline{u}_3, \overline{u}_4$ and $\overline{u}_5$ of
$\mathcal{C}$. It's not hard to see that among any three vertices
from $\overline{u}_1, \overline{u}_2, \overline{u}_3,
\overline{u}_4$ and $\overline{u}_5$ there must be one adjacent pair
of $\mathcal{C}$.
\end{proof}
\begin{proposition}\label{prop:cub-no 10-2}
Let $v_{i},v_{j}\in V(C_{5})$.

$(a)$ If $d_{\mathscr{C}}(v_{i},v_{j})=1$, then $g(v_{i},v_{j})\in
E(C_5)$;

$(b)$ If $d_{\mathscr{C}}(v_{i},v_{j})=3$, then
$g(v_{i},v_{j})\notin E(C_5)$;

$(c)$ If $d_{\mathscr{C}}(v_{i},v_{j})=2$ and $g(v_{i},v_{j})$ is a
diagonal of a square face of $\mathcal{C}$, then
$g(v_{i},v_{j})\notin E(C_5)$.
\end{proposition}
\begin{proof}
$(a)$ We fix two vertices $v_{1}, \ v_{2}$  such that
$v_{1},v_{2}\in V(C_{5})$ and $d_{\mathscr{C}}(v_{1},v_{2})=1$.
Suppose the contrary that $g(v_{1},v_{2})=v_1v_2 \notin E(C_5)$.
Then there is a vertex of $\mathcal{C}$, say $u$, such that it is
adjacent to both $v_{1}$ and $ v_{2}$ in $C_{5}$. Let the five
neighbors of $v_{1}$, $v_{2}$ in $\mathcal{C}$  be $u_{i}$, $i=1, 2,
3, 4, 5$, as shown in Fig.  \ref{fig-5}. By the Fact it is easy to
see that  $u\neq u_{i} \ (i=1, 2, 3, 4, 5)$. Now denote the five
remained vertices of $\mathcal {C}$ by $\overline{u}_i$ ($i=1, 2, 3,
4, 5$). For the sake of convenience, let $\eta_i$ denote the value
of the smaller angle formed by  $v_{1}$, $\overline{u}_i$ and
$v_{2}$ on $\mathcal{C}$. Since $\eta_1 \leq
\frac{\pi}{4}+\frac{\pi}{12}=\frac{\pi}{3}$, $\eta_3
\leq\frac{\pi}{12}+\frac{\pi}{3}+\frac{\pi}{12}=\frac{\pi}{2}$,
$\eta_5 \leq \frac{\pi}{4}+\frac{\pi}{12}=\frac{\pi}{3}$, by the
Fact we have $u \notin \{\overline{u}_1, \overline{u}_3,
\overline{u}_5\}$. Noticing that $\eta_2 \leq
\frac{\pi}{3}+\frac{\pi}{12}+\frac{\pi}{3}$ =
$\frac{\pi}{6}+\frac{\pi}{2}+\frac{\pi}{12}=\frac{3\pi}{4}$, we may
assume that $u=\overline{u}_2$. Then by the Fact we have
$\frac{2\pi}{3} < \eta_2 < \pi$, and therefore
$\eta_2=\frac{3\pi}{4}$ (by Lemma \ref{lem1}).  Now let $v_1'$ be
the other adjacent vertex of $v_1$ in $C_5$. In order to ensure
$\frac{2\pi}{3}<\angle \overline{u}_2v_1v_1'<\pi$, that is, $ \angle
\overline{u}_2v_1v_1' \in \{\frac{3\pi}{4}, \frac{5\pi}{6},
\frac{11\pi}{12}\}$, it is easy to check that $g(v_1,v_1')$  always
intersects $g(v_2,\overline{u}_2)$,  a contradiction. Similarly, if
$u=\overline{u}_4$ we also obtain a contradiction.

\begin{figure}[htbp]
\begin{center}
 \includegraphics*[height=5cm]{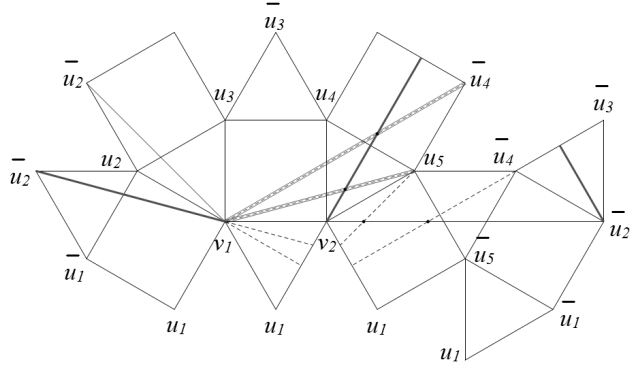}
 \end{center}
 \caption{ An adjacent pair of $\mathcal{C}$. }
\label{fig-5}
\end{figure}

$(b)$  Without loss of generality, we assume that $v_i=u$, $v_j=c$,
as shown in Fig.  \ref{fig-3}. Now suppose the contrary that
$g(v_i,v_j)=g(u,c)\in E(C_5)$. For the sake of convenience, denote
by $u'$ the other adjacent vertex of $u$ in $C_5$. Since
$\frac{2\pi}{3}<\angle u'uc<\pi$, by Lemma \ref{lem1}, we have
$\angle u'uc\in \{\frac{3\pi}{4}, \frac{5\pi}{6},
\frac{11\pi}{12}\}$. There are two cases to consider.

$Case\ 1$. $g(u,c)$ passes through two triangular faces and one
square face.

We consider the rightmost geodesic $g(u,c)$ in Fig.  \ref{fig-3}. If
$\angle u'uc=\frac{3\pi}{4}$ or $\angle u'uc=\frac{11\pi}{12}$, then
$u'= b_3$ or $u'= b_4$. Clearly,  $d_{\mathscr{C}}(b_3,c) =
d_{\mathscr{C}}(b_4,c)=1$. By $(a)$, we  have $ g(b_3,c)\in E(C_5)$
or $ g(b_3,c)\in E(C_5)$. Thus we obtain a 3-cycle $ucb_3u$ or
$ucb_4u$, a contradiction. If $\angle u'uc=\frac{5\pi}{6}$, then
$u'=c$. We obtain a 2-cycle $ucu$, a contradiction again.

$Case\ 2$. $g(u,c)$ passes through one triangular face and two
square faces.

We consider the second leftmost geodesic $g(u,c)$ (the vertical one)
in Fig.  \ref{fig-3}. If $\angle u'uc=\frac{5\pi}{6}$, then $u'=c$
and we obtain a 2-cycle $ucu$, a contradiction. If $\angle u'uc \in
\{\frac{3\pi}{4}, \frac{11\pi}{12}\}$, then $u'=b_5$ or $u'= b_6$.
If $u'= b_6$, by  $(a)$, we obtain a contradiction; If $u'=b_5$, let
$u''$ be the other adjacent vertex of $u'$ in $C_5$. Thus $\angle
uu'u''\in \{\frac{3\pi}{4}, \frac{5\pi}{6}, \frac{11\pi}{12}\}$. If
$\angle uu'u''=\frac{3\pi}{4}$,  then $u''=b_2$ and the geodesic
$g(b_5,b_2)$ in $E(C_5)$ must pass through the faces $\alpha_1,
\alpha_2, \alpha_3$, which intersects $g(u,c)$ in its interior. This
is impossible in $C_5$. If $\angle uu'u''=\frac{5\pi}{6}$, then
$u''=c$ and we obtain a 3-cycle $ub_5cu$, a contradiction again. If
$\angle uu'u''=\frac{11\pi}{12}$,  then $u''=b_2$. Let $u'''$ be the
other adjacent vertex of $u''$ in $C_5$. We know that $\angle
u'u''u''' \in \{\frac{3\pi}{4}, \frac{5\pi}{6}, \frac{11\pi}{12}\}$
which implies that $u''' \in \{a_3, b_5, u\}$. Clearly, $u'''\neq
b_5$ and $u'''\neq u$. If $u'''=a_3$, then $g(b_2, a_3)$ in $E(C_5)$
intersects $g(u, c)$ in its interior, which is a contradiction.

$(c)$ Without loss of generality, let  $v_i=u_1$, $v_j=u_2$ and
$g(u_1,u_2)$ is a diagonal  of a square face of $\mathcal{C}$, as
shown in Fig.  \ref{fig-5}. Now suppose the contrary that
$g(u_1,u_2)\in E(C_5)$. Let $u_1'$ be the other adjacent vertex of
$u_1$ in $C_5$. By the Fact and Lemma \ref{lem1}, we have $\angle
u_1'u_1u_2\in \{\frac{3\pi}{4}, \frac{5\pi}{6}, \frac{11\pi}{12}\}$.
If $\angle u_1'u_1u_2 \in \{\frac{3\pi}{4},
 \frac{11\pi}{12}\}$, then
$u_1'=\overline{u}_3$. Clearly,
$d_{\mathscr{C}}(u_1,\overline{u}_3)=3$, which contradicts to $(b)$.
If $\angle u_1'u_1u_2=\frac{5\pi}{6}$, then $u_1'=u_5$. Let $u''$ be
the other adjacent vertex of $u_1'$ in $C_5$. Noticing that
$d_{\mathscr{C}}(u_1,u_5)=2$ and $g(u_1,u_5)$ is a diagonal of a
square face of $\mathcal{C}$, by the above discussion we know that
if $\angle u_1u_5u_1''\in \{\frac{3\pi}{4}$, $\frac{11\pi}{12}\}$,
then there is a contradiction;  if $\angle
u_1u_5u_1''=\frac{5\pi}{6}$, then $u_1''=\overline{u}_3$. Repeating
the above process again and we obtain a 4-cycle
$u_1u_5\overline{u}_3u_2u_1$, which is a contradiction.
\end{proof}
\begin{proposition}\label{prop:cub-no 10-3}
$C_5$ has only one possible configuration as shown in Fig.
\ref{fig-6}.
\end{proposition}

\begin{figure}[htbp]
\begin{center}
 \includegraphics*[height=5cm]{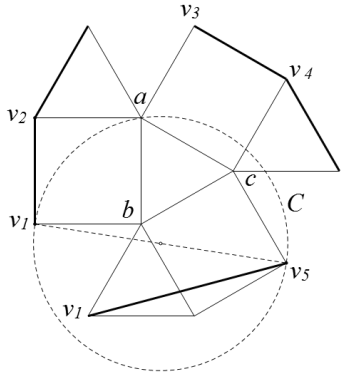}
 \end{center}
 \caption{ The 5-cycle $C_5$. }
\label{fig-6}
\end{figure}

\begin{proof}
Denote the vertices of $ C_5$ by $v_i$  ($i=1, 2, 3, 4, 5$), and
$v_i$, $v_{i+1}$ are adjacent in $C_5$ ($i+1$ takes modulo 5). By
Proposition  \ref{prop:cub-no 10-1} and \ref{prop:cub-no 10-2},
there are two edges of $C_5$, say $e_1$ and $e_2$, which are edges
of $\mathcal{C}$. There are two cases to consider.

$Case\ 1$. $e_1$ and $e_2$ are adjacent in $\mathcal{C}$.

Suppose that $e_1=v_1v_2$ and $e_2=v_2v_3$. Without loss of
generality, we may assume that $v_1=a_1$, $v_2=u$, as shown in Fig.
\ref{fig-3}. Please keep in mind that any angle formed by two
adjacent edges of $C_5$ is between $\frac{2\pi}{3}$ and $\pi$. Thus
we have  $v_{3}=a_3$. By Proposition \ref{prop:cub-no 10-2}, we know
that $d_{\mathscr{C}}(v_{3}, v_{4})=1$, or $d_{\mathscr{C}}(v_{3},
v_{4})=2$ and $g(v_{3}, v_{4})$ passes through one triangular face
and one square face, which implies that $v_4=b_4$ or $c$.

If $v_4=b_4$ and $d_{\mathscr{C}}(v_{4}, v_{5})=1$, then  $v_5=c$
and we obtain a 5-cycle $a_1ua_3b_4ca_1$, as the configuration
described in Fig.  \ref{fig-6}. If $v_4=b_4$ and
$d_{\mathscr{C}}(v_{4}, v_{5})=2$, then $v_5=b_1$. Since
$d_{\mathscr{C}}(v_{5}, v_{1})=1$, by Proposition \ref{prop:cub-no
10-2}$(a)$, $g(v_5, v_1)=v_5v_1\in E(C_5)$. Thus we obtain a 5-cycle
$a_1ua_3b_4b_1a_1$. If $v_4=c$ and $d_{\mathscr{C}}(v_{4},
v_{5})=1$, then $v_5=b_1$, we obtain a 5-cycle $a_1ua_3cb_1a_1$. If
$v_4=c$ and $d_{\mathscr{C}}(v_{4}, v_{5})=2$, then $v_5=a_1$, which
is a contradiction.

 $Case\ 2$. $e_1$ and $e_2$ are not adjacent in $\mathcal{C}$.

Suppose that $e_1=v_1v_2$ and $e_2=v_3v_4$. Without loss of
generality, we may assume that $v_1=a_1$, $v_2=u$, as shown in Fig.
\ref{fig-3}. By Proposition \ref{prop:cub-no 10-2}, we know that
$d_{\mathscr{C}}(v_{2}, v_{3})=1$, or $d_{\mathscr{C}}(v_{2},
v_{3})=2$ and $g(v_{2}, v_{3})$ passes through one triangular face
and one square face. That is, $v_3=a_3$ or $b_4$. If $v_3=a_3$, then
the discussion is same to that in $Case\ 1$. If $v_3=b_4$, then
clearly $v_4=c$. This situation has been discussed in Case 1.
\end{proof}

Now we are back to the proof of Lemma \ref{lem:cub-no 10}. Clearly,
a  5-cycle $C_{5}$ described above  decomposes $\mathcal{C}$ into
two regions, and one of them is shown in Fig.  \ref{fig-6}. Without
loss of generality,
 let $v_6$ be the vertex of the acute triangulation $\mathscr{T}$ lying in this region.
 Since  $\angle v_6v_2v_1<\frac{\pi}{2}$, $v_6$ can not lie in the triangular face $av_2v_3$ and the square face $av_3v_4c$
 except for the edge $ac$.
  Further, since $\angle v_6v_1v_2<\frac{\pi}{2}$, $\angle v_6v_3v_4<\frac{\pi}{2}$ and
  $\angle v_6v_4v_3<\frac{\pi}{2}$, $v_6$ must lie in the  triangular face $abc$. Clearly, $v_6 \notin \{a, b, c\}$ and the
  edge $g(v_1,v_6)$ of $\mathscr{T}$ must intersect the square face $v_1bav_2$.  In Fig. \ref{fig-6}, let
   $C$ be the planar circle with diameter $v_{1}v_{5}$ (here $v_{1}v_{5}$ is the dash segment
    instead of the geodesic). It is easy to see that $v_6$ lies in the interior of the
    upper semi-disc bounded by $C$ and the dash segment $v_1v_5$. As a result, we have
     $\angle v_1v_6v_5>\frac{\pi}{2}$, which contradicts to the fact that $\mathscr{T}$ is an acute triangulation. Therefore,
there is no acute triangulation of $\mathcal{C}$ with ten triangles.
\end{proof}

 Combining Theorem \ref{thm1}, \ref{thm2}, Lemma \ref{lem:cub-no
8}, \ref{lem:cub-no 10}, we obtain the following main theorem
immediately.
\begin{theorem}\label{thm-cub}
The surface of the cuboctahedron  admits an acute triangulation with
$12$ triangles, and there is no acute triangulation with fewer
triangles.
\end{theorem}
\paragraph{Acknowledgements.}
 The second  author gratefully
acknowledges financial supports by NSF of China (10701033,
10426013);  program for New Century Excellent Talents in University,
Ministry of Education of China; the Plan of Prominent Personnel
Selection and Training for the Higher Education Disciplines in Hebei
Province; and WUS Germany (Nr. 2161). She is also indebt to Beijing
University for the financial support during her academic visit
there.

%


\begin{thebibliography}{99}
\bibitem{bura} Y. D. Burago, V. A. Zalgaller, Polyhedral embedding
 of a net (Russian), {\it Vestnik Leningrad. Univ.} {\bf 15} (1960) 66-80.

\bibitem{ccgl} C. Cassidy and G. Lord, A square acutely
triangulated, {\em J. Recr. Math.} {\bf 13} (1980/81) 263-268.

\bibitem{convex-qua} M. Cavicchioli, Acute triangulations of convex
quadrilaterals, {\it Discrete Appl. Math.}, {\bf 160} (2012)
1253-1256.


\bibitem{margard1} M. Gardner, Mathematical games, A fifth
collection of ~``brain-teasers", {\it Sci. Amer.} {\bf 202} (2)
(1960) 150-154.

\bibitem{margard2} M. Gardner, Mathematical games, The games and puzzles of Lewis Carroll, and the answers
 to February's problems, {\it Sci.
Amer.} {\bf 202} (3) (1960) 172-182.

\bibitem{margard} M. Gardner, { \em New Mathematical Diversions},
 Mathematical Association of America, Washington  D.C., 1995.

\bibitem{hiz01} T. Hangan, J. Itoh and T.
Zamfirescu, Acute triangulations, {\em Bull. Math. Soc. Sci. Math.
Roumanie} {\bf 43} (91) No. 3$-$4 (2000) 279-285.

\bibitem{itoh-yuan} J. Itoh, L. Yuan, Acute triangulations of flat tori, {\em Europ. J.
Comb.} {\bf 30} (2009) 1-4.

\bibitem{izam20} J. Itoh, T. Zamfirescu, Acute triangulations of the regular dodecahedral
surface  {\em Europ. J. Comb.} {\bf 28} (2007) 1072-1086.

\bibitem{izam12} J. Itoh, T. Zamfirescu, Acute
triangulations of the regular icosahedral surface, { \em Discrete
Comput. Geom.} {\bf 31} (2004) 197-206.

\bibitem{hmarpoly} H. Maehara, Acute triangulations of polygons, {\em Europ. J.
Comb.}  {\bf 23} (2002) 45-55.

\bibitem{hmarqua} H. Maehara, On acute
triangulations of quadrilaterals, { \em Proc. JCDCG 2000; Lecture
Notes Comp. Sci.} {\bf 2098} (2001) 237-354.

\bibitem{hmar-polyhedral} H. Maehara, On a proper acute triangulation of a polyhedral
surface, {\em Discrete Math.}, {\bf 311(17)} (2011), 1903-1909.

\bibitem{wam} W.  Manheimer, Solution to Problem E1406:
Dissecting an obtuse triangle into acute triangles, {\em Amer. Math.
Monthly} {\bf 67} (1960) 923.

\bibitem{saraf}  S. Saraf, Acute and nonobtuse triangulations of polyhedral surfaces, {\it Europ. J.
Comb.} {\bf 30} (2009)  833-840.

\bibitem{yuan-pentagon} L. Yuan, Acute triangulations of pentagons, {\it Bull. Math. Soc. Sci. Math. Roumanie}, {\bf 53}(101)   (2010)  393-410.

\bibitem{yuanpolygon}  L. Yuan, Acute triangulations of polygons,
{ \em Discrete Comput. Geom.}  {\bf 34} (2005) 697-706.

\bibitem{yuan-trapezoid} L. Yuan, Acute triangulations of trapezoids, {\it Discrete Appl. Math.}, {\bf 158} (2010)  1121-1125.
\bibitem{yuan-mobius}  L. Yuan, T. Zamfirescu, Acute triangulations of Flat M\"{o}bius strips,
{ \em Discrete Comput. Geom.}  {\bf 37} (2007) 671-676.

\end{thebibliography}
\end{document}